\documentclass[11pt, reqno]{amsart}

\numberwithin{equation}{section}

\usepackage{amsmath}
\usepackage{amsthm}
\usepackage{amssymb}
\usepackage{amsfonts}
\usepackage{latexsym}
\usepackage{graphicx}

\newtheorem{Proposition}{Proposition}
\newtheorem{Corollary}{Corollary}
\newtheorem{Theorem}{Theorem}

\newtheorem{Lemma}{Lemma}

\theoremstyle{definition}
\newtheorem*{Definition}{Definition}
\newtheorem{Example}{Example}
\newtheorem*{Remark}{Remark}

\renewcommand{\phi}{\varphi}

\newcommand{\A}{\mathcal{A}}

\newcommand{\N}{\mathbb{N}}

\newcommand{\R}{\mathbb{R}}
\newcommand{\K}{\mathbf{K}}
\newcommand{\p}{\mathbf{P}}
\newcommand{\s}{\mathbf{S}}
\newcommand{\y}{\mathbf{y}}
\newcommand{\z}{\mathbf{z}}
\newcommand{\ba}{\mathbf{a}}
\newcommand{\x}{\mathbf{x}}
\newcommand{\w}{\mathbf{w}}
\newcommand{\pp}{\mathbf{p}}
\newcommand{\uu}{\mathbf{u}}
\newcommand{\vv}{\mathbf{v}}

\newcommand{\bsal}{b.s.a.l.}

\begin{document}

\title{Positivity and optimization for
semi-algebraic functions}

    \author{Jean B. Lasserre}
    \address{LAAS-CNRS  and Institute of Mathematics \\
     University of Toulouse, France}
    \email{lasserre@laas.fr}
    \urladdr{http://www.llas.fr/~lasserre}
    \thanks{}

    \author{Mihai Putinar}
    \address{   Department of Mathematics\\
            University of California at Santa Barbara\\
            Santa Barbara, California\\
            93106-3080}
    \email{mputinar@math.ucsb.edu}
    \urladdr{http://math.ucsb.edu/\textasciitilde mputinar}    
    \thanks{Work partially supported by a National Science Foundation Grant-USA. The authors are grateful to the American Institute of Mathematics in Palo Alto, CA, which has hosted a meeting where the present work was initiated}

    \subjclass[2000]{}

    \begin{abstract} We describe algebraic certificates of positivity for functions
    belonging to a finitely generated algebra of Borel measurable functions, with particular emphasis
    to algebras generated by semi-algebraic functions. In which case the standard global optimization
    problem with constraints given by elements of the same algebra is reduced via  a natural change of variables to the better understood case of polynomial optimization. A collection of simple examples and numerical experiments complement the theoretical parts of the article.
    	
    \end{abstract}

\maketitle

\section{Introduction}

 Key sums of squares representation results
from real algebraic geometry have been successfully applied during the last decade to
polynomial optimization. Indeed, for instance, by combining such
results with semidefinite programing (a powerful tool of convex optimization),
one may approximate as closely as desired 
(or sometimes compute exactly) the global minimum of a polynomial over a compact basic semi-algebraic set. Several other extensions have also been proposed for finding real zeros of polynomials equations, minimizing a rational function, and solving the {\em generalized problem of moments} with polynomial data. For more details on this so-called {\em moment-sos} approach the interested reader is referred to e.g. \cite{lasserrebook} and the many references therein.

Even though questions involving polynomial data represent a sufficiently large framework,
one may wonder whether this approach can be extended to a more general class of problems
that involve some larger algebra than the polynomial ring. Typical examples are algebras
generated by the polynomials and some elementary functions like the absolute value,
trigonometric polynomials, exponentials, splines or discontinuous step functions.

To mention only one classical source of such natural generalizations
we invoke the original works of Tchebysheff and Markov concerning the limiting values of integrals of 
specific measures
with a finite number of prescribed moments against non-polynomial functions.  A variety of remarkable results: first on the theoretical side \cite{Krein,KN}, then in statistics \cite{Isii1,Isii2,KS1,KS2},
numerical integration \cite{BP}
and not last in the theory of best approximation \cite{Z,KS2} have resulted from Tchebysheff-Markov
work. All these researches, plus quite recent
contingent works, such as wavelet decomposition of functions, involve non-polynomial moment problems. In this respect, our essay provides an algebraic analysis of positivity in the ``pre-dual'
of concrete distribution/measure spaces. The importance of non-polynomial solutions to classical
problems in the polynomial algebra, such as Hilbert's 17-th Problem, was early recognized, see for instance
\cite{Delzell1,Delzell2} and also \cite{BCR}.

{\bf Contribution.}
The purpose of this paper is to povide a first step in the understanding of Positivestellens\"atze
in algebras generated by non-polynomials functions, and to provide a robust relaxation method
for the numerical verification of the constrained positivity of such functions as well as for optimization purposes.

(a) First we consider the case of an algebra $\A$ of functions generated by a finite family
of functions that contains the polynomials. In this context we prove a Positivstellensatz
for a function positive on a compact set $\K\subset\R^n$ defined by finitely many inequality constraints involving functions of this algebra. When $\A$ is the polynomial ring the set
$\K$ is a basic semi-algebraic set and one retrieves a  
Positivstellensatz due the second author \cite{Put}.

(b) Further on we consider a slightly different framework. 
The algebra $\A$ now consists of functions generated by
basic monadic and dyadic relations on polynomials. The 
monadic operations are $\vert\cdot\vert$ and $(\cdot)^{1/p}$, $p=1,2,\ldots$,
while the
dyadic operations are $(+,\times,/,\wedge,\vee)$. Notice that this algebra contains 
highly nonlinear and nondifferentiable functions!
We then show that this algebra is a subclass
of semi-algebraic functions and every element of $\A$
has the very nice property that it has a {\it lifted basic semi-algebraic representation}.
As in (a) we provide an extension to this algebra of the
Positivstellensatz due the second author for polynomials \cite{Put}.

(c) Finally, in both cases (a) and (b) we also provide
a converging  hierarchy of semidefinite relaxations for optimization. 
Actually, the original problem in the algebra reduces to an equivalent {\it polynomial} optimization problem in a lifted space. The dimension of the lifting (i.e. the number of additional variables)
is directly related to the number of elementary operations
on polynomials needed to define the functions in the description of the original problem.
Moreover, the resulting polynomial optimization problem exhibits a natural sparsity pattern 
with respect to the lifted variables for which the important {\it running intersection property} holds true.
Therefore, if on the one hand one has to handle additional variables, on the other hand
this increase in dimension is partly compensated by this sparsity patterns that permits to
apply the sparse semidefinite relaxations defined in \cite{Waki} and whose convergence was proved in \cite{lasserresparse}.
To the best of our knowledge, the present article outlines a first systematic generalization of
 the moment-sos approach  initiated in \cite{Lasserre0,Lasserre1,Parrilo} for polynomial optimization
to algebras of non-polynomial functions and to (global) semi-algebraic optimization.\\

{\it Acknowledgement.} The authors thank Tim Netzer and Murray Marshall for
a careful reading of the article and for their constructive comments.\\

\section{Preliminaries and main result} Although a more general abstract measure space or probability space
framework is very natural for the main results of the present article, due to the optimization theory applications, we confine ourselves to
an algebra of Borel measurable functions defined on a set of the Euclidean space. 

We start by recalling a few standard definitions from real algebra, \cite{PD}. Specifically, let
$X \subset \R^d$ be a Borel measurable set and let $\A$ be a unital algebra of real valued Borel measurable functions defined on $X$. We denote by $\Sigma \A^2$ the convex cone of squares of elements of $\A$.
A {\it quadratic module} $Q \subset \A$ is a convex cone containing $1$, such that
$$ (\Sigma \A^2) \cdot Q \subset Q.$$ In practice we deal with {\it finitely generated quadratic modules},
of the form
$$ Q = \Sigma \A^2 + \Sigma \A^2 \cdot h_1 + ...+ \Sigma \A^2 \cdot h_m,$$
where $h_1,...,h_m \in \A$. The {\it positivity set} of a quadratic module $Q$ is by definition
$$ P(Q) = \{ x \in X\,:\: f(x) \geq 0, \ \forall f \in Q \}.$$

Similarly to complex algebraic geometry, the duality between a quadratic module and its positivity set
lies  at the heart of real algebraic geometry \cite{BCR, PD}. An useful tool for implementing this duality
is provided by the abstract moment problem on the algebra $\A$, from where we import a simple definition: we say that the quadratic module $Q \subset \A$ possesses the {\it moment property}
if every linear functional $L \in \A'$ which is non-negative on $Q$ is represented by a positive
Borel measure $\mu$ supported by $P(Q)$.

Also, for duality type arguments, we recall the following concept: an element $f$ of a convex cone
$C \subset \A$ lies in the {\it algebraic interior} of $C$, if for every $h \in \A$ there exists $\epsilon>0$
with the property $f + t h \in C$ for all $t, 0 \leq t <\epsilon$.

The starting point of all representation theorems below is the following simple but crucial observation.
Unless otherwise stated, all measurable functions and sets below are meant to be Borel mesurable.

\begin{Proposition} Let $\A$ be a unital algebra of measurable functions defined on the
measurable subset $X \subset \R^d$. Let $Q \subset \A$ be a quadratic module with an
algebraic interior point and possessing the moment property. 

If a function $f \in \A$ is positive on $P(Q)$, then $f \in Q$.
\end{Proposition}

\begin{proof} Assume by contradiction that $f \notin Q$, and let $\xi$ be an interior point of $Q$.
By the separation theorem for convex sets, see \cite{Ko}, or for a historical perspective
\cite{Kakutani, KR}, there exists a linear functional $L \in \A$ satisfying
$$ L(f) \leq 0 \leq L(h), \ \ h \in Q, \ \ L(\xi)>0.$$
By assumption there exists a positive Borel measure $\mu$, such that
$$ L(a) = \int_{P(Q)}  d\mu, \ \  a \in \A.$$
Since $L(\xi) >0$ we infer that the measure $\mu$ is non-zero. On the other hand,
$f|_{P(Q)} >0$ and $L(p) \leq 0$, a contradiction.
\end{proof}

The converse is also notable:

\begin{Corollary}  Assume $X$ compact and $\A \subset C(X)$ is a subalgebra of continuous functions which separates the points of $X$. Let $Q \subset \A$ be a quadratic module with the constant function $1$ in its
algebraic interior and
suppose that every $f \in \A$ which is positive on $P(Q)$ belongs to $Q$. Then $Q$ has the
moment property.
\end{Corollary}

\begin{proof} Let $L \in \A'$ be a linear functional satisfying $L|_Q \geq 0, \ L(1) >0$. Due to the compactness assumption all continuous functions are comparable 
to the constant function $\xi =1$ in the following precise sense: for every $F \in C(X)$, there is a
positive constant $C$ with the property
$ -C \xi(\x) \leq F(\x) \leq  C \xi(\x)$. Then Marcel Riesz extension theorem \cite{MRiesz} provides a
linear functional $\Lambda \in C(X)'$ which extends $L.$ Hence $\Lambda$ is
represented by a signed measure $\mu$; since the algebra $\A$ is dense in $C(X)$ by
Stone-Weierstrass theorem, and $\Lambda(a^2) \geq 0$ for all $a\in \A$,
we infer that $\mu$ is a positive measure.

In addition,  we know by assumption that
for every point $x \notin P(Q)$ there exists
$h \in Q$ such that $h(x)<0$ and $h|_{P(Q)} \geq 0$. Hence
$$ \int_X h a^2 d\mu \geq 0, \ \ \ a \in A.$$
Again Stone -Weierstrass Theorem implies that $x \notin {\rm supp}(\mu)$, that is
the representing measure $\mu$ is
supported by the closed set $P(Q)$.
\end{proof}

Thus the main questions we are faced with at this first stage of inquiry are: under which conditions a quadratic module $Q$
has an interior point and/or possesses the moment property. While the first question has a simple
solution, having to do with the boundedness of the positivity set $P(Q)$, the second
one involves solving an abstract moment problem, it is more delicate, but on the other hand
has a long and glorious past, with a wealth of partial results and side remarks  \cite{KN,KS1,KS2}.

Towards finding solutions to the above two questions we consider only a finitely generated
algebra $ \A = \R [h_1,...,h_n]$, where $h=(h_1,...,h_n)$ is an n-tuple of measurable functions on the set
$X \subset \R^d$. Let $\y=(y_1,...,y_n)$ be variables, so that, by noetherianity
$$ \A \cong \R [\y]/I,$$
where $I$ is a finitely generated ideal of $\R[\y]$. 
The ideal $I$ describes all the algebraic relations between generators of the algebra.
This ideal is in addition {\it radical}, in the following
sense: if $p_1^2 +...+ p_k^2 \in I$, where $p_1,...,p_k \in \R[\y]$, then $p_1,...,p_k \in I$, \cite{BCR}.
Indeed, if $p_1(h)^2 +...+ p_k(h)^2 =0$ as functions defined on $X$, then $p_1 \circ h=0, ..., p_k \circ h =0$
in the algebra $\A$. Denote by "$\circ$" the usual composition of functions, i.e.,
with $g:\R^n\to\R$ and $f:\R\to\R$, 
$\x\mapsto (f\circ g)(\x):=f(g(\x))$.

\begin{Lemma} Assume that $ \A = \R [h_1,...,h_n]$ is a finitely generated algebra of measurable functions and let $Q \subset \A$ be a quadratic module. If $1-(h_1^2 +...+h_n^2) \in Q$, then
the constant function $1$ belongs to the algebraic interior of $Q$.
\end{Lemma}

For a simple algebraic proof we refer to \cite{PD}. Already a first approximative solution to the moment problem
associated to $Q$ is available.

\begin{Proposition} Let $ \A = \R [h_1,...,h_n]$ be an algebra of measurable functions
defined on the set $X \subset \R^d$ and let $L \in \A'$ be a linear functional which is non-negative on the
quadratic module $\Sigma \A^2 + (1-(h_1^2 +...+h_n^2) )\Sigma \A^2$. Then there exists a positive measure $\nu$, supported by the unit ball in $\R^n$, with the property
$$ L(p(h_1,...,h_n)) = \int_{\| \y \| \leq 1} p(\y) d\nu(\y), \ \  p \in \R[\y].$$
\end{Proposition}

\begin{proof}   The familiar Gelfand-Naimark-Segal construction can be invoked at this moment.
In short, define an inner product on the polynomial algebra $\R[\y]$ by
$$ \langle p, q \rangle = L( p(h) q(h)), \ \ p,q \in \R[y].$$ Denote by $J$ the set of null-vectors
$ p \in J$ if and only if $\langle p, p \rangle = 0.$ It is an ideal of $\R[\y]$  by Cauchy-Schwarz
inequality. The quotient space $\R[\y]/J$ is endowed then with a non-degenerate
inner product structure. Let $H$ be its Hilbert space completion. The multiplication 
operators with the variables $M_i p (\y) = y_i p(\y)$ are self-adjoint and bounded, due to the
positivity assumption imposed on $L$:
$$ \langle M_i  p , M_i p \rangle = L( h_i^2 p(h)^2) \leq L(p(h)^2) = \langle p, p \rangle.$$
In addition, the operators $M_i$ mutually commute on the Hilbert space $H$. Thus the spectral theorem
gives a positive Borel measure $\nu$, with the property
$$ L(p(h)) = \langle p, 1 \rangle = \langle p(M_1,...,M_n)1, 1 \rangle = \int p d\nu, \ \ p \in \R[\y].$$
In addition, the support of the measure $\nu$ is contained in the unit ball, as a consequence of the
spherical contraction assumption $M_1^\ast M_1 +...+M_n^\ast M_n \leq I.$
\end{proof}

Thus, a quadratic module $Q \subset \A = \R [h_1,...,h_n]$ with $(1-(h_1^2 +...+h_n^2) )\in Q$
satisfies the moment property if the measure $\nu$ appearing in the above proposition
is the push-forward via the map $h : X \longrightarrow \R^d$ of a positive measure $\mu$ supported by $P(Q)$, that is:
\[\int_{\| \y \| \leq 1} p(\y) d\nu(\y) =\,\int_{P(Q)} p(h(\x)) d\mu(\x), \ \  p \in \R[\y]\,.\]
A series of sufficient conditions assuring $\nu = h_\ast \mu$ in the above sense are discussed next.

Again for $\A = \R [h_1,...,h_n]$ consider a finitely generated quadratic module
$Q = \Sigma \A^2 + g_0(h) \Sigma \A^2+ g_1(h) \Sigma \A^2+...+ g_m(h) \Sigma \A^2$,
where $g_0,...,g_m \in \R[\y]$ and $g_0(\y) = 1-y_1^2-...-y_n^2.$ Let also the radical ideal
$I = (p_1(\y),...,p_k(\y))$. Then
$$\tilde{Q} = \Sigma \R[\y]^2 + g_0(\y) \Sigma \R[\y]^2  + ...+ g_m(\y) \Sigma \R[\y]^2 
\pm p_1  \Sigma \R[\y]^2 \pm ...\pm p_k  \Sigma \R[\y]^2 $$
is a quadratic module in the polynomial algebra $\R[\y]$. Its positivity set
$P(\tilde{Q})$ is contained in the closed unit ball intersection with the zero set $V(I)$
of the ideal $I$. Note that we have a natural pull-back (substitution) map
$$ h^\ast : \R[\y] \longrightarrow \A, \ \ h^\ast (p) = p\circ h, \ \ p \in \R[\y],$$
and $h^\ast(\tilde{Q}) = Q$. At the level of positivity sets we obtain
$$ h (P(Q)) \subset P(\tilde{Q}),$$ but the inclusion might be strict. Indeed, we derive the following
direct observation.

\begin{Lemma} With the above notation, assume that $f = p\circ h \in Q$, where $p \in \R[\y]$.
Then $p$ is non-negative on $P(\tilde{Q})$.
\end{Lemma}

\begin{proof} The assumption $f \in Q$ amounts to the algebraic identity
$$ f\circ h = \sigma\circ h + (g_0\circ h)\,( \sigma_0\circ h)  + \cdots+ (g_m\circ h )\,(\sigma_m\circ h),$$
where $\sigma, \sigma_0,...,\sigma_m \in \Sigma \R[\y]^2.$ In its turn, this identity becomes
$$ f(\y) =  \sigma(\y) + g_0(\y) \sigma_0(\y)  + ...+ g_m(\y) \sigma_m(\y) + r(\y),$$
where $r \in I$. By evaluating at a point $\ba \in P(\tilde{Q}) \subset V(I)$ we infer
$$ f(\ba) =  \sigma(\ba) + g_0(\ba) \sigma_0(\ba)  + ...+ g_m(\ba) \sigma_m(\ba) \geq 0.$$

\end{proof}

We will see in the next section simple examples showing that the positivity of
a function $f$ on $P(Q)$ does not guarantee in general a decomposition of the form
$f \in Q$. The gap between the two statements is explained in the following general result.

\begin{Theorem} 
\label{thmain}
Let $\A = \R [h_1,...,h_n]$ be a finitely generated algebra of measurable functions
defined on a measurable set $X \subset \R^d$ and let $Q$ be a quadratic module
with the function $1$ in its algebraic interior. Assume that $f \in \A$ is positive
on $P(Q)$. Then $f \in \hat{Q}$, where $\hat{Q}$ is the pullback by the map
$h : X \longrightarrow \R^d$ of a quadratic module which serves as
a positivity certificate for the set $h(P(Q))$.
\end{Theorem}

Of course the choice of $\hat{Q}$ is not unique, and it may not be a finitely
generated quadratic module. For instance, as a first rough approximation,
one can choose
$$ Q' = \{ p\circ h; \ p|_{h(X) \cap P(\tilde{Q})} \geq 0 \},$$
where $\tilde{Q}$ is the pull-back of $Q$ via the projection map
$\pi : \R[\y] \longrightarrow \A.$

\begin{Corollary} 
\label{coro1}Assume, in the conditions of Theorem \ref{thmain}, that the set $h(X)$
is closed in $\R[y]$. Choose an at most countable set of polynomials
$v_i \in \R[\y], \ i \in J, \ | J | \leq \infty,$ with the property
$$ \{ \y \in \R^n\,: \ v_i(\y) \geq 0, \ i \in J \} = h(X) \cap P(\tilde{Q}).$$
Then every element $f\circ h \in \A$ which is positive on the set $P(Q)$
can be written as
$$ f\circ h = \sigma \circ h + (M- h_1^2-...-h_n^2)\,(\sigma_0\circ h)  + 
\sum_{i \in J_0} (v_i\circ h) (\sigma_i\circ h),$$
where $M>0$, the subset $J_0 \subset J$ is finite and $\sigma, \sigma_0, \sigma_i \in \Sigma \R[\y]^2,
\ i \in J_0.$
\end{Corollary}

The remaining of this section
is concerned with particular choices of the
"saturation" process $Q \mapsto \hat{Q}$ and the resulting effective sums of squares representations.

\subsection{Examples in algebras of semi-algebraic functions}
Let $X=\R^n$ and let $\A = \R[\x,h_1(\x), ..., h_t(\x)]$ for some functions $h_i:\R^n\to\R$,
$i=1,\ldots,t$. Consider the set $\K\subset\R^n$ defined by:
\begin{equation}
\label{setk}
\K\,:=\,\{\x\in\R^n\::\: g_j(\x,h_1(\x),\ldots,h_t(\x))\geq 0,\:j=1,\ldots,m\},
\end{equation}
and suppose that $\K$ is compact and 
\[\x\mapsto M-\Vert \x\Vert^2-\sum_{i=1}^t \vert h_i(\x)\vert^2\,>0\,\quad\mbox{on }\K.\]
\begin{Example}
\label{ex1}
Here $h_i(\x)=\vert x_i\vert$, $i=1,\ldots,n$ (so let
$\vert\x\vert=(\vert x_i\vert)$.  In this case 
$\y=(y_1,\ldots,y_{2n})$ and one has the algebraic relation
$y^2_{n+i}=y^2_i$, $i=1,\ldots,n$, between the generators, hence $\A \cong \R[\y]/I$, with 
$I:=\langle y^2_{n+1}-y^2_1,\ldots,y^2_{2n}-y^2_{n}\rangle$. 
Let's choose the quadratic module
with the constant $1$ in its algebraic interior:
\begin{eqnarray*}
Q &=& \Sigma \A^2 
+ (M-(\sum_{i=1}^nx_i^2 +|x_i|^2))\Sigma \A^2  +\sum_{j=1}^m\Sigma \A^2 g_j(\x,\vert\x\vert)\\
&=&  \Sigma \A^2 + (M-2\Vert\x\Vert^2) \Sigma \A^2+\sum_{j=1}^m\Sigma \A^2 g_j(\x,\vert\x\vert)
\end{eqnarray*}
The lifted quadratic module is
$$ \tilde{Q} = \Sigma \R[\y]^2 + (M-\Vert\y\Vert^2) \Sigma \R[\y]^2+
\sum_{j=1}^m\Sigma \R[\y]2 g_j(\y),$$
whose positivity support set is the set
$$\{\y\in\R^{2n}\::\:M-\Vert\y\Vert^2\geq 0;\quad g_j(\y)\geq0,\:j=1\ldots,m\}.$$
On the other hand, $ P(Q)$ lives in $\K\subset\R^n$.

For instance if $\K=[-1,1]^n$ and $t=n$ with $g_j(\x)=1-x_j^2$, $j=1,\ldots,t$,
a function such as $|x_1|^3 + 1/2$ belongs to the algebra $\A$, is positive on $P(Q)$, but it cannot be
written as 
$$ |x_1|^3 + 1/2 = \sigma (\x,|\x|) + (2n-2\Vert\x\Vert^2)\sigma_0(\x,|\x|)
+\sum_{j=1}^n\sigma_j(\x,\vert\x\vert)g_j(\x),$$
with $\sigma, \sigma_j\in \Sigma \R[\y]^2$. Indeed, such a representation would lift to
$$ y_{n+1}^3 + 1/2 = \sigma (\y) + (2n-\Vert \y\Vert^2)\sigma_0(\y)+
\sum_{j=1}^n\sigma_j(\y)g_j(\y),$$
and this is obviously impossible by choosing for example $y_i=0$ for all
$i\neq n+1$ and $y_{n+1}=-1$.

In order to obtain the correct representation we invoke the main result of the previous section,
and first consider the image set
$$ h(X)\,=\,\{ (\x,|\x|); \ \x \in \R^n \} \subset \R^{2n}.$$
This set can be exactly described by the inequalities
$$ y_{n+i}^2 - y_i^2 \geq 0, \ y_{n+i}^2 - y_i^2 \geq 0, \ y_{n+i} \geq 0;\quad i=1,\ldots,n.$$ 
Therefore, 
\begin{eqnarray*}
h(X)\cap P(\tilde{Q})\,=\,\{\y\in\R^{2n}&:&2n-\Vert \y\Vert^2\geq 0;\:g_j(\x)\geq0,\:j=1,\ldots,t;\\
&&y_{n+i}^2-y_i^2=0,\:y_{n+i}\geq0,\:i=1,\ldots, n\},\end{eqnarray*}
and by Corollary \ref{coro1}, the quadratic module in $\R[\y]$
generated by the polynomials that describe the above semi-algebraic set
$h(X)\cap P(\tilde{Q})$, provide the desired certificate
of positivity. And so we have:

\begin{Proposition} Let $\K$ be as in (\ref{setk}).
An element $f$ of the algebra $\R[\x,|\x|]$ which is positive on 
$\K$ can be written as
\begin{eqnarray*}
 \x\mapsto f(\x,\vert \x\vert) &=&\sigma(\x,|\x|) + (n-\Vert\x\Vert^2)\sigma_0(\x,|\x|) \\
&&+ \sum_{i=1}^n\vert x_i\vert \, \sigma_i(\x,|\x|)
+\sum_{j=1}^tg_j(\x,\vert\x\vert) \, \psi_j(\x,|\x|)\end{eqnarray*}
where $\sigma, \sigma_i, \psi_j \in \Sigma \R[\y]^2$.
\end{Proposition}
\end{Example}

\begin{Example}
\label{ex2}
Given  two polynomials $p,q\in\R[\x]$,
let $\A:=\R[\x,p(\x)\vee q(\x)]$ where $a\vee b:=\max[a,b]$, and
let $\K$ be as in (\ref{setk}) with $t=1$ and $h_1=p(\x)\vee q(\x)$.
Recall that $2 (a\vee b)=\vert a-b\vert+(a+b)$, and so 
one may consider the new algebra $\A'=\R[\x,\vert p(\x)-q(\x)\vert]$ since a polynomial
in the variables $\x$ and $p(\x)\vee q(\x)]$ is a particular polynomial in $\x$ and $\vert p(\x)-q(\x)\vert$. In this case $\y=(y_1,\ldots,y_{n+1})$ and one has the algebraic dependency
\[y_{n+1}^2\,=\,(p(\x)-q(\x)) ^2\]
and the additional constraint $y_{n+1}\geq0$.

\begin{Proposition} Let $\K$ be as in (\ref{setk}) and let
$f$ be an element 
of the algebra $\mathcal{A}:=\R[\x,p(\x)\vee q(\x)]]$ which is positive on $\K$.
Equivalently
\[\x \mapsto f(\x,p(\x)\vee q(\x)])\,=\,\tilde{f}(\x,\vert p(\x)-q(\x)\vert),\]
for some $\tilde{f}\in\mathcal{A}':=\R[\x,\vert p(\x)-q(\x)\vert]$, positive
on $\K$. Then $\tilde{f}$ can be written
\begin{eqnarray*}
\tilde{f}(\x,\vert p(\x)-q(\x)\vert)&=&\sigma(\x,\vert p(\x)-q(\x)\vert)\\
&&+(M-\Vert \x\Vert^2-\vert p(\x)-q(\x)\vert^2)\sigma_0(\x,\vert p(\x)-q(\x)\vert)\\
&& + \vert p(\x)-q(\x)\vert \sigma_1(\x,\vert p(\x)-q(\x)\vert)\\
&& + \sum_{j=1}^m\psi_j(\x,\vert p(\x)-q(\x)\vert)g_j(\x,\vert p(\x)-q(\x)\vert)
\end{eqnarray*}
where $\sigma, \sigma_i, \psi_j \in \Sigma \R[y_1,\ldots,y_{n+1}]^2$.
\end{Proposition}
Of course a similar thing can be done with $p(\x)\wedge q(\x):=\min[p(\x),q(\x)]$ using 
$2(a\wedge b)=(a+b)-\vert a-b\vert$.
\end{Example}

\begin{Example}
\label{ex3}
Let $\uu=(u_i)_{i=1}^n\in\R[\x]^n$ and let $f$ be an element 
of the algebra $\mathcal{A}:=\R[\x,\Vert \uu(\x)\Vert_p]]$ which is positive on $\K$,
where 
$\Vert \x\Vert_p=(\sum_{i=1}^n\vert x_i\vert^p)^{1/p}$, with either $p\in\N$ or $p^{-1}\in\N$.

If $p/2\in\N$ one uses the algebraic lifting $\y\in\R^{n+1}$ with:
\[y_{n+1}^p=\sum_{i=1}^nu_i(\x)^p;\quad  y_{n+1}\geq 0.\]
If $p=2q+1$ with $q\in\N$, one uses the algebraic lifting $\y\in\R^{2n+1}$, with:
\[y_{n+i}^2=u_i(\x)^2;\quad y_{2n+1}^p=\sum_{i=1}^ny_{n+i}^p;\quad y_{n+i}\geq 0,
\:i=1,\ldots,n+1.\]
If $1/p\in\N$ one uses the algebraic lifting:
\[y_{n+i}^{2/p}=u_i(\x)^2;\quad y_{2n+1}=\left(\sum_{i=1}^ny_{n+i}\right)^{1/p};\quad y_{n+i}\geq 0,\:i=1,\ldots,n+1.\]
\begin{Proposition} Let $\K$ be as in (\ref{setk}) and let
$f$ be an element 
of the algebra $\mathcal{A}:=\R[\x,\Vert \uu(\x)\Vert_p]$ which is positive on $\K$.

{\rm (a)} If $p/2\in\N$ then $f$ can be written
\begin{eqnarray*}
f(\x,\Vert \uu(\x)\Vert_p)&=&\sigma(\x,\Vert \uu(\x)\Vert_p)\\
&&+(M-\Vert \x\Vert^2-\Vert \uu(\x)\Vert_p^2)\,\sigma_0(\x,\Vert \uu(\x)\Vert_p)\\
&& + \Vert \uu(\x)\Vert_p\,\sigma_1(\x,\Vert \uu(\x)\Vert_p) + \sum_{j=1}^mg_j(\x)\,\psi_j(\x,\Vert \uu(\x)\Vert_p)
\end{eqnarray*}
where $\sigma, \sigma_i, \psi_j \in \Sigma \R[y_1,\ldots,y_{n+1}]^2$.

{\rm (b)} If $p=2q+1$ with $q\in\N$ then $f$ can be written
\begin{eqnarray*}
f(\x,\Vert \uu(\x)\Vert_p)&=&\sigma(\x,\vert \uu(x)\vert,\Vert \uu(\x)\Vert_p)\\
&&+(M-\Vert \x\Vert^2-\Vert \uu(\x)\Vert^2-\Vert \uu(\x)\Vert_p^2)\,\sigma_0(\x,\vert \uu(\x)\vert,\Vert \uu(\x)\Vert_p)\\
&& + \sum_{i=1}^n\vert u_i(\x)\vert\,\sigma_i(\x,\vert \uu(\x)\vert,\Vert \uu(\x)\Vert_p)\\
&&+ \Vert \uu(\x)\Vert_p\,\phi(\x,\vert \uu(\x)\vert,\Vert \uu(\x)\Vert_p)\\
&& + \sum_{j=1}^mg_j(\x)\,\psi_j(\x,\vert \uu(\x)\vert,\Vert \uu(\x)\Vert_p)
\end{eqnarray*}
where $\sigma, \sigma_i, \psi_j,\phi \in \Sigma \R[y_1,\ldots,y_{2n+1}]^2$.

{\rm (c)} If $1<p^{-1}\in\N$ then $f$ can be written
\begin{eqnarray*}
f(\x,\Vert \uu(\x)\Vert_p)&=&
\sigma(\x,\vert u_1(x)\vert^p,\ldots,\vert u_n(x)\vert^p,\Vert \uu(\x)\Vert_p)\\
&&+(M-\Vert \x\Vert^2-\sum_{i=1}^n\vert u_i(\x)\vert^{2p}-\Vert \uu(\x)\Vert_p^2)\\
&&\times\sigma_0(\x,\vert u_1(x)\vert^p,\ldots,\vert u_n(x)\vert^p,\Vert \uu(\x)\Vert_p)\\
&& + \sum_{i=1}^n\vert u_i(\x)\vert^p\,
\sigma_i(\x,\vert u_1(x)\vert^p,\ldots,\vert u_n(x)\vert^p,\Vert \uu(\x)\Vert_p)\\
&& + \Vert \uu(\x)\Vert_p\,\phi(\x,\vert u_1(x)\vert^p,\ldots,\vert u_n(x)\vert^p,\Vert \uu(\x)\Vert_p)\\
&& + \sum_{j=1}^mg_j(\x)\,\psi_j(\x,\vert u_1(x)\vert^p,\ldots,\vert u_n(x)\vert^p,\Vert \uu(\x)\Vert_p)
\end{eqnarray*}
where $\sigma, \sigma_i, \phi,\psi_j \in \Sigma \R[y_1,\ldots,y_{2n+1}]^2$.
\end{Proposition}
\end{Example}

\begin{Example}
\label{exsqrt}
Let $\A = \R[\x,\sqrt{p(\x)}]$. Of couse one here considers the new basic semi-algebraic set:
\begin{equation}
\label{newsetk}
\K'\,:=\,\K\,\cap\,\{\x\::\:g_{m+1}(\x)\geq0\}\quad\mbox{with}\quad\x\mapsto g_{m+1}(\x):=p(\x),
\end{equation}
and the lifting
\[y_{n+1}^2\,=\,p(\x);\quad y_{n+1}\geq 0.\]
\begin{Proposition} Let $f$ be an element 
of the algebra $\mathcal{A}:=\R[\x,\sqrt{p(\x)}]$ which is positive on $\K'$.
Then $f$ can be written
\begin{eqnarray*}
f(\x,\sqrt{p(\x)})&=&\sigma(\x,\sqrt{p(\x)})
+(M-\Vert \x\Vert^2-p(\x))\,\sigma_0(\x,\sqrt{p(\x)})\\
&& + \sum_{j=1}^{m+1}g_j(\x)\,\psi_j(\x,\sqrt{p(\x)}) +\sqrt{p(\x)}\,\phi(\x,\sqrt{p(\x)})\\
\end{eqnarray*}
where $\sigma, \sigma_i, \phi,\psi_j \in \Sigma \R[y_1,\ldots,y_{n+1}]$.
\end{Proposition}
\end{Example}

\subsection{Examples in algebras of nonsemi-algebraic functions}
\begin{Example}
Let $X=[0,\pi/2]$ and $\A = \R[x,\sin x]$. Again $\y=(y_1,y_2)$ but this time
the algebra $\A$ is isomorphic to the full
polynomial algebra $\R[\y]$. For illustration we consider the quadratic module
$Q \subset \A$ generated by the elements $x, \pi/2 -x, 1 - \sin^2 x.$ Although the
inequality $ 1 - \sin^2 x \geq 0$ is redundant, it is necessary to add it for having
the function $1$ in the algebraic interior of $Q$. 

As we wish to obtain a certificate
of positivity for a function $f(x,\sin x)$ belonging to $\A$, 
an algebraic description of the graph of $x\mapsto \sin x$ is in order. Choose the 
polynomials
$$ x\mapsto p_k(x) = \sum_{j=0}^k (-1)^k \frac{x^{2k+1}}{(2k+1)!}, \qquad k \geq 0.$$
It is well known that 
$$ p_{2k}(x) \geq p_{2k+2}(x) \geq  \sin x \geq p_{2k+1}(x) \geq p_{2k-1}(x), \quad
\forall x \in [0, \pi/2],$$
and that these polynomials converge uniformly to $\sin x$ on $[0,\pi/2]$.
A description of the graph of $\sin x$ is:
$$ h(X)\,=\,\{ (x, \sin x)\,:\ 0 \leq x \leq \pi/2 \} = $$ $$
\{ \y\in\R^2\,: \ 0 \leq y_1 \leq \pi/2 , \quad p_{2k+1}(y_1) \leq y_2 \leq p_{2k}(y_1),\quad k \geq 0\}.$$

In view of our main result, the conclusion is:

\begin{Proposition} Let $f$ be an element of the algebra $\R[x,\sin x]$ and assume that
$f>0$ for all $ x \in [0,\pi/2]$. Then there exists $k \geq 0$ such that
$$ x\mapsto f(x,\sin x) = \sigma(x,\sin x) + x \sigma_0(x,\sin x) + (\pi/2 -x) \sigma_2(x,\sin x) + $$ $$[p_{2k}(x)-\sin x]
\sigma_3(x, \sin x) + [\sin x - p_{2k+1}(x)] \sigma_4(x, \sin x),$$
where $\sigma, \sigma_i \in \Sigma \R[\y]^2.$
\end{Proposition}
\end{Example}
\begin{Example}
Let $X=\R$ and $\A = \R[x, e^{ax}, e^{bx}]$. Two distinct cases, corresponding to $a,b$ commensurable
or not, will be treated. This dichotomy is well known to the experts in the control theory of delay systems.

To fix ideas, we assume that $b>a>0$ and the base is the interval $[0,1]$. Denote
$y_1 = x, y_2 = e^{ax}, y_3 = e^{bx}$. We will impose the archimedeanity constraints
$$ x \geq 0, 1 -x \geq 0,\ \  y_2 \geq 0, e^a - y_2 \geq 0, \ \ y_3 \geq 0, e^b -y_3 \geq 0.$$
Along with these we have to consider the polynomial approximations of the two exponentials:
$$ p_n(ax) = \sum_{k=0}^n \frac{a^k x^k}{k!} \leq y_2 \leq  p_n(ax) + \frac{a^{n+1} e^a x^{n+1}}{(n+1)!}$$
and
$$  p_n(bx) = \sum_{k=0}^n \frac{b^k x^k}{k!} \leq y_3 \leq  p_n(bx) + \frac{b^{n+1} e^b x^{n+1}}{(n+1)!}.$$ The reader will recognize above an upper bound of the remainder of Taylor's series approximation
of the exponential function.\\

{\it Case I.\ The frequencies $a,b$ are not commensurable.} Then there is no algebraic
relation among the variables $y_1,y_2,y_3$ and a function $ f \in  \R[x, e^{ax}, e^{bx}]$
which is positive on the interval $[0,1]$ is decomposable into a weighted sum of squares
generated by $x, 1-x, e^{ax}, e^a-e^{ax}, e^{bx}, e^b-e^{bx}, e^{ax} - p_n(ax),  p_n(ax) + \frac{a^{n+1} e^a x^{n+1}}{(n+1)!}-e^{ax}, e^{bx} -p_n({bx}),  p_n(bx) + \frac{b^{n+1} e^b x^{n+1}}{(n+1)!}- e^{bx},$
for a large enough $n$.

Note that the above list of weights is redundant, as for instance the constraint
$e^{ax}>0$ follows from $ e^{ax} - p_n(ax)>0$, and so on. In practice it will be important to
identify a minimal set of positivity constraints.\\

{\it Case II. The frequencies $a,b$ are commensurable.} Assume that $Na = Mb$ where
$N,M$ are positive integers. Then the decomposition above of a positive (on $[0,1]$) element
 $f \in  \R[x, e^{ax}, e^{bx}]$ holds modulo the algebraic relation $y_2^N = y_3^M$.
 
\end{Example}

\begin{Example} When dealing with trigonometric (real valued)
polynomials in $n$ variables, one has to consider the algebra
$\A = \R[ \sin \theta_i, \cos \theta_i], \  1 \leq i \leq n.$ The standard lifting
proposed in the present article is
$$ y_i = \sin \theta_i, \ \  z_i = \cos \theta_i,\ \  1 \leq i \leq n,$$
modulo the ideal generated by the relations
\begin{equation}\label{torus}
y_i^2 + z_i^2 -1 = 0, \ \  1 \leq i \leq n.\end{equation}
Note that in this case every quadratic module $Q \subset \A$
in the variables
$\sin \theta_i, \cos \theta_i$ lifts to an archimedean quadratic module $\hat{Q}$
in the variables $y_i, z_i$. Specifically, relations (\ref{torus}) assure that
the constant function $1$ belongs to the algebraic interior of $\hat{Q}$.
\end{Example}

\section{Representation of positive semi-algebraic functions}

We now consider an algebra of functions which forms an important subclass of 
semi-algebraic functions.

Let $(g_j)_{j=1}^m\subset\R[\x]$ and let
$\K\subset\R^n$ be the basic semi-algebraic set
\begin{equation}
\label{setkk}
\K\,:=\,\{\x\in\R^n\::\: g_j(\x)\geq 0,\:j=1,\ldots,m\}.
\end{equation}

Recall that $f\,:\K\to\R$ is a semi-algebraic function if its graph 
$\Psi_f:=\{(\x,f(\x))\,:\,\x\in\K\}$
is a semi-algebraic set of $\R^n\times\R$.

Recall the notation
\[a\vee b=\max[a,b];\quad a\wedge b:=\min [a,b];\quad\vert \x\vert =(\vert x_1\vert,\ldots,\vert x_n\vert)\in\R^n,\]
and let $\A$ be the algebra of fonctions $f\,:\K\to\R$
generated by finitely many of the dyadic operations 
$\{+,\times,/,\vee,\wedge\}$ 
and monadyc operations $\vert\cdot\vert$ and $(\cdot)^{1/p}$ ($p\in\N$)
on polynomials.

For instance, with $u,v,\ell\in\R[\x]$ and $h\in\R[\x]^2$, $q\in\N$, the function
\begin{equation}
\label{demo}
\x\mapsto f(\x):=(\vert u(\x)+v(\x)\vert ^{1/p})\wedge\left(\sqrt{\Vert h(\x)\Vert_{2q+1}+1/\ell(\x)}\right)\end{equation}
is a member of $\A$ (assuming of course that $\ell(\x)>0$ or $\ell(\x)<0$
for all $\x\in\K$).

\begin{Definition}
A function $f\in\A$ is said to have a {\it basic semi-algebraic 
lifting} (in short a "\bsal"), or $f$ is {\it basic semi-algebraic} (b.s.a.) if there exist $p,s\in\N$, polynomials $(h_k)_{k=1}^s\subset\R[\x,y_1,\ldots,y_p]$
and a basic semi-algebraic set
\[\K'_f:=\{(\x,\y)\in\R^{n+p}\::\: \x\in\K;\:h_k(\x,\y)\geq0,\quad k=1,\ldots,s\}\]
such that the graph of $f$ (denoted $\Psi_f$) satisfies:
\begin{equation}
\Psi_f:=\{(\x,f(\x))\::\:\x\in\K\}\,=\,\{(\x,y_p)\::\:(\x,\y)\in\K'_f\}.
\end{equation}
In other words,
$\Psi_f$ is an orthogonal projection of the basic semi-algebraic set $\K'_f$
which lives in the lifted space $\R^{n+p}$.
\end{Definition}
Hence by the projection theorem of real algebraic geometry,
a function $f\in\A$ that admits a \bsal~ is semi-algebraic.
\begin{Lemma}
\label{sal}
Let $\K$ be the basic semi-algebraic set in (\ref{setk}).
Then every well-defined $f\in\A$ has a basic semi-algebraic lifting.
\end{Lemma}
\begin{proof}
It is obvious that the sum $f+g$ and multiplication $fg$
of two b.s.a. functions $f,g\in\A$, is b.s.a.
So let $f\in\A$ be b.s.a., i.e.,
\[\Psi_f=\{(\x,y_p)\::\:(\x,\y)\in\K'_f\subset\R^{n+p}\}\]
for some integer $p$ and some basic semi-algebraic set $\K'_f$.
Then:
\begin{itemize}
\item If $f\neq 0$ on $\K$ then $f^{-1}$ has the \bsal
\[\Psi_{f^{-1}}=\{(\x,y_{p+1})\::\:(\x,(y_1,\ldots,y_p))\in\K';\:y_{p+1}y_p=1\}\]
and so $f/g$ is b.s.a. whenever $f$ and $g$ are.
\item With $f\geq0$ on $\K$ and $q\in\N$, $f^{1/q}$ has the \bsal
\[\Psi_{f^{1/q}}=\{(\x,y_{p+1})\::\:(\x,(y_1,\ldots,y_p))\in\K';\:y_{p+1}^{q}=y_p;\:y_{p+1}\geq 0\}\]
\item $\vert f\vert$ has the \bsal
\[\Psi_{\vert f\vert}=\{(\x,y_{p+1})\::\:(\x,(y_1,\ldots,y_p))\in\K';\:y_{p+1}^2=y_p^2;\,y_{p+1}\geq0\}\]
\item If $f$ and $g$ have a \bsal then so does $f\wedge g$ (resp. $f \vee g$) because
$2(f\wedge g)=(f+g)-\vert f-g\vert$
(resp. $2(f\vee g)=(f+g)+\vert f-g\vert$).
\end{itemize}
\end{proof}
\begin{Example}
\label{exlift}
Let $f\in\A$ be the function defined in (\ref{demo}), and let
$\K'_f\subset\R^{n+8}$ be the basic semi-algebraic set defined by the constraints:
\begin{eqnarray*}
\K'_f=\left\{(\x,\y)\in\R^{n+9}\::\quad\right.g_j(\x)&\geq&0,\quad j=1,\ldots,m;\\
y_i^2-h_i(\x)^2&=&0;\quad y_i\geq0;\quad i=1,2.\\
y_3^{2q+1}-y_1^{2q+1}-y_2^{2q+1}&=&0;\quad y_3\geq 0\\
y_4\,\ell(\x) &=&1;\\
y_5^2-y_4-y_3&=&0;\quad y_5\geq0\\
y_6^{2p}-(u(\x)+v(\x))^2&=&0;\quad y_6\geq0\\
y_7^2-(y_6-y_5)^2&=&0;\quad y_7\geq 0\\
2y_8-(y_6+y_5)+y_7&=&0\left.\right\}
\end{eqnarray*} 
Then $\Psi_f=\{(\x,y_8)\::\:(\x,\y)\in\K'_f\}$. 
\end{Example}
Observe that with each variable $y_k$ of the lifting $\y$ is associated 
a certain function $v_k\in\A$. For instance in Example \ref{exlift}:
\[y_i\to v_i(\x):=\vert h_i(\x)\vert,\:i=1,2;\:y_3\to v_3(\x)
:=(v_1(\x)^{2q+1}+v_2(\x)^{2q+1})^{1/(2q+1)}\]
\[y_4\to v_4(\x):=(\ell(\x)^{-1};\:y_5\to v_5(\x):=\sqrt{v_3(\x)+v_4(\x)}\]
\[y_6\to v_6(\x):=\vert u(\x)+v(\x)\vert^{1/p};\:y_7\to v_7(\x):=\vert v_6(\x)-v_5(\x)\vert\]
and $y_8\to v_8(\x):=f(\x)$.
\vspace{0.2cm}

Next, with $\K\subset\R^n$ as in (\ref{setkk}), consider the set $\s\subseteq\K$
defined by:
\begin{equation}
\label{sets}
\s:=\{\x\in\K\,:\,h_\ell(\x)\geq 0,\:\ell=1,\ldots,s\}
\end{equation}
for some finite family $(h_k)_{k=1}^s\subset\A$. 
\begin{Lemma}
\label{lemma-setS}
The set $\s$ in (\ref{sets}) is a semi-algebraic set which is the projection of a 
lifted basic semi-algebraic set. It can be written
\begin{eqnarray}
\nonumber
\s\,=\,\left\{\x\in\R^{n}\::\:\exists \y\in\R^t\right.&\mbox{s.t.}&\x\in\K;\quad u_k(\x,\y)=0,\quad k=1,\ldots,r\\
\label{setS}
&&\left. y_j\geq0,\quad j\in J\right\}
\end{eqnarray}
for some integers $r,t$, some polynomials $(u_k)_{k=1}^r\subset\R[\x,\y]$ and some
index set $J\subseteq\{1,\ldots,t\}$.

With the lifting $\y$ is associated a vector
of functions $\vv\in\A^t$.
\end{Lemma}
\begin{proof}
By Lemma \ref{sal}, each function $h_\ell$ is b.s.a. and so there is an integer $n_\ell$
and a basic
semi-algebraic set $\K'_\ell\subset\R^{n+n_\ell}$ such that:
\[\Psi_{h_\ell}\,=\,\{(\x,y^{\ell}_{n_\ell})\::\:(\x,\y^\ell)\in\K'_\ell\},\quad \ell=1,\ldots,s,\]
with $\y^\ell\in\R^{n_\ell}$, $\ell=1,\ldots,s$.
Therefore,
\[\s\,=\,\{\x\in\R^n\::\:
\exists \,\y^1,\ldots,\y^s\mbox{ s.t. }(\x,\y^\ell)\in \K'_\ell;\:y^\ell_{n_\ell}\geq0;\quad \ell=1,\ldots,s\},\]
and so $\s$ is the projection of a lifted basic semi-algebraic set 
which lives in $\R^{n+\sum_\ell n_l}$.
Obviously it is of the form (\ref{setS}).
\end{proof}
Next we state the main result of this section.
\begin{Theorem}
Let $\K\subset\R^n$ be a compact basic semi-algebraic set as in (\ref{setkk}) and 
let $\s$ be as in (\ref{sets}) for some finite family $(h_\ell)\subset\A$.
Let $f\in\A$ have the \bsal
\begin{eqnarray*}
\Psi_f\,=\,\left\{(\x,z_{m_f})\right.&:&\x\in\K;\:q_\ell(\x,\z)\,=\, 0,\quad \ell=1,\ldots,t_f\\
&&\left. z_j\geq 0,\quad j\in J_f\right\}\end{eqnarray*}
(where $\z=(z_1,\ldots,z_{m_f})$ and $J\subseteq\{1,\ldots,m_f\}$) and let 
$\pp=(\pp_1,\ldots,\pp_{m_f})\in\A^{m_f}$ be the vector of functions associated with
the lifting $\z$.

If $f$ is positive on $\s$ then
\begin{eqnarray}
\nonumber
f(\x)&=&\sigma_0(\x,\vv(\x),\pp(\x))+\sum_{j=1}^m\sigma_j(\x,\vv(\x),\pp(\x))\,g_j(\x)\\
&&+\sum_{k\in J}\psi_k(\x,\vv(\x),\pp(\x))\vv_k(\x)\\
&&+\sum_{\ell\in J_f}\phi_\ell(\x,\vv(\x),\pp(\x))\,\pp_\ell(\x)\\
\label{put}
&&+\phi_0(\x,\vv(\x),\pp(\x))(M-\Vert (\x,\vv(\x),\pp(\x))\Vert^2)
\end{eqnarray}
for some $(\sigma_j),(\psi_k),(\phi_\ell)\subset\Sigma\R[\x,\y,\z]$ and a sufficiently large $M>0$.
\end{Theorem}

\begin{Example} The following simple example illsutrates why
the b.s.a. functions are much nicer than general semi-algebraic functions.
Let $\A = \R[x, \chi_S]$ be the algebra generated by the polynomials in the variables 
$x \in \R$ and the characteristic function of a basic algebraic set, say for simplicity given by a single
inequality
$$ S = \{ x \in \R;\ h(x) \geq 0 \}$$
where $h \in \R[x]$ is a non-zero polynomial. An element $f \in \A$ is of the form
$f = f_1 \chi_S + f_2 (1-\chi_S),$ with $f_1,f_2 \in \R[x]$. Let $Q \in \A$ be an archimedean
quadratic module with associated positivity set $P(Q)$. The question whether
$f|_{P(Q)}>0$ reduces therefore to $(f_1)|_{P(Q)\cap S} >0$ and independently\\
$(f_2)|_{P(Q) \setminus S}>0$. While the first condition is reducible to the above setting
$$ (f_1)|_{P(Q + h \Sigma \R[\x]^2)}>0,$$
the second condition involves the positivity of the polynomial $f_2$ on a {\it non-closed}
semi-algebraic set. This case can be resolved in general only with the help of Stengle's
Positivstellensatz, see for details \cite{BCR,PD}.

To be more explicit, consider $h(x) = \|x\|^2 -1$ and $Q = \Sigma \R[x]^2+
(2-\|x\|^2)\Sigma \R[x]^2$. We are then led to consider the problem
\begin{equation}\label{Stengle}
 f_2(x)>0 \ \ \ \ {\rm whenever} \ \ \ \ \|x\|^2<1,
 \end{equation}
where $f_2$ is a polynomial in $\x$. Stengle's Theorem states that 
the implication (\ref{Stengle}) holds if and only if there exists a positive integer
$N$ and $\sigma_i \in \Sigma \R[\x]^2, \ \ 0 \leq i \leq 3,$ such that
$$ f_2 (\sigma_0 + (1-\|x\|^2)\sigma_1 ) = (1-\|x\|^2)^{2N} + \sigma_2 + (1-\|x\|^2)\sigma_3.$$ 
In conclusion, even if the graph of the characteristic function $\chi_S$ is semi-algebraic,
a departure from the requirement that it is a (lifted) basic semi-algebraic set implies that
the sought certificate of positivity of an element $f \in \A$ requires denominators.

\end{Example}

\section{Semi-algebraic optimization}

In this section we consider the application of
preceding results to optimization of a function $f$
of the algebra $\A$, over a compact set $\s$ as in (\ref{sets}) 
(with $\K$ as in (\ref{setkk})) and whose
defining inequality constraints are also defined with functions in this algebra $\A$.
Recall that by Lemma \ref{lemma-setS}, the set $\s$ can also be written as the projection
(\ref{setS}) of a basic semi-algebraic set in a lifted space $\R^{n+t}$.

Consider the optimization problem:
\begin{equation}
\p:\quad f^*:=\:\min_\x\: \{\,f(\x)\::\: \x\in\s\:\}
\end{equation}
\begin{Proposition}
With $\K$ in (\ref{setkk}) compact and with $\s$ as in (\ref{sets})-(\ref{setS}),
let $f\in\A$ have the \bsal
\begin{eqnarray*}
\Psi_f\,=\,\left\{(\x,z_f)\right.&:&\x\in\K;\:q_\ell(\x,\z)\,=\, 0,\quad \ell=1,\ldots,t_f\\
&&\left. z_j\geq 0,\quad j\in J_f\right\}\end{eqnarray*}
(where $\z=(z_1,\ldots,z_{m_f})$ and $J\subseteq\{1,\ldots,m_f\}$) and let 
$\pp=(\pp_1,\ldots,\pp_{m_f})\in\A^{m_f}$ be the vector of functions associated with
the lifting $\z$. Then $\p$ is also the polynomial optimization:
\[f^*\,=\,\min_{\x,\y,\z}\:\{z_{m_f}\::\: (\x,\y,\z)\,\in\,\Omega\}\]
where $\Omega\subset\R^{n+t+m_f}$ is the basic semi-algebraic set
\begin{equation}
\label{setomega}
\begin{array}{lrll}
\Omega=&\{(\x,\y,\z)\::\:&g_j(\x)&\geq0,\quad j=1,\ldots,m\\
&&u_k(\x,\y)&=0,\quad k=1,\ldots,r\\
&&q_\ell(\x,\z)&=0,\quad \ell=1,\ldots,t_f\\
&&y_k,z_j&\geq0,\quad k\in J,\:j\in J_f\:\}
\end{array}
\end{equation}
\end{Proposition}
The proof follows directly from Lemma \ref{lemma-setS} and the definiton
of the \bsal $\Psi_f$.\\

Then one may apply the moment-sos approach developed in e.g. \cite{Lasserre1,lasserrebook}
and build up a hierarchy of semidefinite programs whose associated sequence of optimal values converges to the global optimum $f^*$.

That is, with $\w=(\w_{\alpha\beta\gamma})$ being a real sequence indexed in the index set $\N^n\times\N^{t}\times\N^{m_f}$, let $L_\w:\R[\x,\y,\z]\to\R$ be the linear functional
\[p\:(=\sum_{\alpha\beta\gamma}p_{\alpha\beta\gamma}\,\x^\alpha\,\y^\beta\;\z^\gamma)\quad\mapsto
L_\w(p):=\sum_{\alpha\beta\gamma}p_{\alpha\beta\gamma}\,w_{\alpha\beta\gamma}.\]
Let $\theta\in\R[\x,\y\,\z]$ be the quadratic polynomial
\[(\x,\y\,\z)\,\mapsto\,\theta(\x,\y,\z)\,:=\,M-\Vert (\x,\y,\z)\Vert^2,\]
with $M$ such that $M>\Vert(\x,\y,\z)\Vert^2$ on $\Omega$,
and let
$a_k:=\lceil ({\rm deg}\,u_k)/2\rceil$ and
$b_\ell:=\lceil ({\rm deg}\,q_\ell)/2\rceil$, for every $k=1,\ldots,r$, $\ell=1,\ldots,t_f$.

For $i\geq\max_{k,\ell}[a_k,b_\ell]$, consider the semidefinite program
\begin{equation}
\label{sdp1}
\begin{array}{rlll}
\rho_i=&\min &w_{00m_f}&\\
&\mbox{s.t.}&M_i(\w)&\succeq0\\
&&M_{i-a_k}(u_k\,\w)&\succeq0,\quad k=1,\ldots,r\\
&&M_{i-b_\ell}(q_\ell\,\w)&\succeq0,\quad \ell=1,\ldots,t_f\\
&&M_{i-1}(y_j\,\w)&\succeq0,\quad j\in\,J\\
&&M_{i-1}(z_j\,\w)&\succeq0,\quad j\in\,J_f\\
&&M_{i-1}(\theta\,\w)&\succeq0\\
&&w_{000}&=1.
\end{array}
\end{equation}
\begin{Theorem}
\label{th-sdp}
Consider the hierarchy of semidefinite programs (\ref{sdp1}).

{\rm (a)} The sequence $(\rho_i)$ is monotone nondecreasing and 
$\rho_i\,\to\,f^*$ as $i\to\infty$.

{\rm (b)} Let $\w^i$ be an optimal solution of (\ref{sdp1}) and assume that
\begin{equation}
\label{rank}
{\rm rank}\,M_i(\w)\,=\,{\rm rank}\,M_{i-c}(\w)\,=:d,\end{equation}
whre $c:=\max_{b_k,a_\ell}$. Then $\rho_i=f^*$ and
$\w$ is the moment sequence of a $d$-atomic probability measure supported 
on $d$ points of the set $\Omega$ defined in (\ref{setomega}).
\end{Theorem}
\begin{proof}
As $\p$ is a standard polynomial optimization problem,
the proof follows directly from e.g. \cite{Lasserre1,lasserrebook}.\end{proof}

Theorem \ref{th-sdp} is illustrated below on the following two simple examples:
\begin{Example}
Let $n=2$ and $\p:\:f^*=\displaystyle\max_\x\{\,\vert x_1\vert \,x_2-x_1^2\,:\:
\Vert \x\Vert^2=1\}$. 
Hence
\[\Omega=\{(\x,z)\::\:\Vert \x\Vert^2-1=0;\:z^2-x_1^2=0;\:z\geq 0\}.\]
Observe that $\vert x_1\vert\leq 1$ on $\K=\{\x\,:\,\Vert\x\Vert^2=1\}$ and so one may choose
$\theta(\x,z):=\:3-\Vert(\x,z)\Vert^2$.
At the second relaxation (\ref{sdp1}) the rank condition (\ref{rank}) is satisfied with $d=2$,
and one obtains $f^*=\rho_2=0.2071$ with the two optimal solutions:
\[(\x^*,z^*)=(0.3827,\,0.9239,\,0.3827);\:
(\x^*,z^*)=(-0.3827,\,-0.9239,\,0.3827).\]
\end{Example}
\begin{Example}
Let $n=2$ and $\p:\:f^*=\displaystyle\max_\x\{\,x_1\,\vert x_1-2x_2\vert \,:\:
\Vert \x\Vert^2=1\}$. 
Hence
\[\Omega=\{(\x,z)\::\:\Vert \x\Vert^2-1=0;\:z^2-(x_1-2x_2)^2=0;\:z\geq 0\}.\]
Observe that $\vert x_1-2x_2\vert\leq 3$ on $\K=\{\x\,:\,\Vert\x\Vert^2=1\}$ and so one 
may choose $\theta(\x,z):=11-\Vert(\x,z)\Vert^2$.
At the second relaxation (\ref{sdp1}) the rank condition (\ref{rank}) is satisfied with $d=1$,
and one obtains $f^*=\rho_2=1.6180$ and the optimal solution:
$(\x^*,z^*)=(0.8507,\,-0.5257,\,1.9021)$.
\end{Example}

\begin{Remark}

(a) A usual lifting strategy (well-known from optimizers) to handle optimization problems 
of the form $\p:\min_\x\{f(\x)\,:\,\x\in\K\}$,
with $f(\x)=\max_{j\in J}f_j(\x)$ for some finite index set $J$, is to use only one lifting $z$ and solve
\[\min_{\x,z}\:\{\,z\,:\,\x\in\K,\:z\geq f_j(\x),\:j\in J\}.\]
Similarly with $f(\x)=\vert g(\x)\vert$ one uses $z\geq\pm g(\x)$.
In this case one has only one additional variable with no equality constraint. 
In particular, this lifting preserves convexity (when present). But this lifting
is valid in some particular cases only and for instance it does not work if $f$ is replaced with
$f(\x)=h(\x)-\max_{j\in J}[f_j(\x)]$, or $f(\x)=h(\x)-\vert g(\x)\vert$.

(b) In the description of the set $\Omega$, a "lifting" $\y^\ell$ is associated with each non polynomial 
function $h_\ell\in\A$ that appears in the description of $\p$, as well as a lifting $\z$ for $f$ 
(if non polynomial).
This can be rapidly penalizing. However it is worth noting that
there is no mixing between the lifting variables $\z$ and $\y^\ell$, as well as 
between the lifting variables $\y^\ell$ of different $\ell$'s. The coupling of all these variables is through the variables $\x$. Hence there is an obvious sparsity pattern if the 
set of all variables $(\x,\y,\z)$ is written as the union (with overlaps)
\[(\x,\y,\z)\,=\,(\x)\,\cup\,(\x,\z)\,\cup\,(\x,\y^1)\,\cup\,\cdots\,\cup\,(\x,\y^s)\]
and this sparsity pattern obviously satifies the co-called {\it running intersection property}, 
because for each $1\leq k <s$,
\[(\x,\y^{k+1})\,\bigcap\,\left((\x)\cup(\x,\z)\displaystyle\cup_{i=1}^{k}(\x,\y^i)\right)\,\subseteq (\x).\]
See e.g. \cite{lasserresparse}.
This is a good news because with such a sparsity pattern one may use
the sparse semidefinite relaxations defined in \cite{Waki},
whose associated sequence of optimal values was still proved
to converge to  $f^*$ in \cite{lasserresparse}. Hence
 if on the one hand lifting is penalizing, on the other hand
the special structure of the lifted polynomial optimization
problem permits to use specialized "sparse" relaxations, which (partly) compensates
the increase in the number of variables. For more details 
on sparse semidefinite relaxations, the interested reader is referred to e.g. \cite{lasserresparse}
and \cite{Waki}. And finally, one should also bear in mind that the original problem $\p$ being
very hard, there is obviously some price to pay for its resolution!

\end{Remark}

\end{document}